\documentclass[12pt]{amsart}
\usepackage{amssymb,amscd,amsmath}
\usepackage{amsaddr}
\usepackage{fullpage}
\usepackage{tikz-cd}
\usepackage{todonotes}
\usepackage{verbatim} 
\usetikzlibrary{intersections}

\usepackage{amsthm}
\theoremstyle{definition}
\newtheorem{theorem}{Theorem}
\newtheorem{lemma}[theorem]{Lemma}
\newtheorem{proposition}[theorem]{Proposition}

\theoremstyle{definition}
\newtheorem{definition}[theorem]{Definition}

\theoremstyle{remark}
\newtheorem{remark}[theorem]{Remark}

	\newcommand{\C}{\mathbb{C}}
	\newcommand{\Z}{\mathbb{Z}}
	
	\newcommand{\h}{\mathcal{H}}

	\newcommand{\SL}{\mathrm{SL}}

	\DeclareMathOperator{\Aut}{Aut}

    \renewcommand{\mod}{\text{ mod}\,}
	\newcommand{\abcd}[4]{\begin{pmatrix}#1&#2\\#3& #4\end{pmatrix}}
	\newcommand{\sabcd}[4]{\left(\begin{smallmatrix}#1&#2\\#3& #4\end{smallmatrix}\right)}

	\title{Note on Fourier expansions at cusps}
\author{Fran\c{c}ois Brunault}
\email{francois.brunault@ens-lyon.fr}
\address{\'{E}NS Lyon, UMPA, 46 all\'{e}e d'Italie, 69007 Lyon, France}
\author{Michael Neururer}
\email{neururer@mathematik.tu-darmstadt.de}
\address{TU Darmstadt, Schlo\ss gartenstr. 7, 64289 Darmstadt, Germany}
	\thanks{The second author was partially funded by the DFG-Forschergruppe 1920 and the LOEWE research unit ``Uniformized Structures in Arithmetic and Geometry''}
	\usepackage{hyperref}
	\hypersetup{pdfauthor={Fran\c{c}ois Brunault and Michael Neururer},%
	            pdftitle={Note on Fourier expansions at cusps},%
	            pdfsubject={Number Theory},%
	            pdfcreator={XeLaTeX},%
	            colorlinks=true,%
	            linkcolor=[rgb]{0,.6,1},
	            urlcolor = blue,
	            citecolor=blue}

\begin{document}

\begin{abstract}
This was originally an appendix to our paper `Fourier expansions at cusps' \cite{brunaultneururer:FAC}. The purpose of this note is to give a proof of a theorem of Shimura on the action of $\Aut(\C)$ on modular forms for $\Gamma(N)$ from the perspective of algebraic modular forms. As the theorem is well-known, we do not intend to publish this note but want to keep it available as a preprint.
\end{abstract}

\maketitle

In this note we give a new proof of the following theorem which is originally due to Shimura, see \cite[Theorem 8]{shimura1975} and \cite[Lemma 10.5]{shimura2000}. It gives the interaction between the $\SL_2(\Z)$-action and the $\Aut(\C)$-action on spaces of modular forms on the group $\Gamma(N)$. These actions on a modular form $f(\tau)=\sum_n a_n e^{2\pi i n\tau/w}$ of weight $k\geq 1$ are defined as follows:
\[
f| g(\tau) = \frac{1}{(C\tau + D)^k} f\Bigl(\frac{A\tau+B}{C\tau+D}\Bigr), \qquad f^\sigma(\tau) = \sum_n \sigma(a_n) e^{2\pi in\tau/w}
\]
for $g=\sabcd ABCD\in\SL_2(\Z)$ and $\sigma\in\Aut(\C)$. For any integer $N\geq 1$ we denote $\zeta_N=e^{2\pi i/N}$.

\begin{theorem}\cite{shimura1975,shimura2000}\label{fg sigma}
Let $f\in M_k(\Gamma(N))$ be a modular form of weight $k \geq 1$ on $\Gamma(N)$ and $g$ and $\sigma$ be as above such that $\sigma(\zeta_N)=\zeta_N^\lambda$ with $\lambda \in (\Z/N\Z)^\times$. Then
\[
(f|g)^{\sigma} = f^{\sigma} | g_\lambda,
\]
where $g_\lambda$ is any lift in $\SL_2(\Z)$ of the matrix $\sabcd{A}{\lambda B}{\lambda^{-1} C}{D} \in \SL_2(\Z/N\Z)$.
\end{theorem}

This theorem immediately implies that if a modular form $f$ of level $N$ has Fourier coefficients in a field $K_f$, then the Fourier coefficients of $f|\sabcd ABCD$ for any $\sabcd ABCD\in\SL_2(\Z)$ will lie in $K_f(\zeta_N)$. In \cite{brunaultneururer:FAC} we obtain improved results if $f$ is a modular form for $\Gamma_0(N)$ or $\Gamma_1(N)$ and in the case of newforms on $\Gamma_0(N)$, we determine the number field generated by the Fourier coefficients of $f|\sabcd ABCD$ explicitly.

We recall the theory of algebraic modular forms, in order to give a new proof of Theorem \ref{fg sigma}. For more details on this theory, see \cite[Chap. II]{katz-interpolation} and the references therein.

\begin{definition} Let $R$ be an arbitrary commutative ring, and let $N \geq 1$ be an integer. A \emph{test object} of level $N$ over $R$ is a triple $T=(E,\omega,\beta)$ where $E/R$ is an elliptic curve, $\omega \in \Omega^1(E/R)$ is a nowhere vanishing invariant differential, and $\beta$ is a \emph{level $N$ structure} on $E/R$, that is an isomorphism of $R$-group schemes
\begin{equation*}
\beta : (\mu_N)_R \times (\Z/N\Z)_R \xrightarrow{\cong} E[N]
\end{equation*}
satisfying $e_N(\beta(\zeta,0),\beta(1,1)) = \zeta$ for every $\zeta \in (\mu_N)_R$. Here $\mu_N = \operatorname{Spec} \Z[t]/(t^N-1)$ is the scheme of $N$-th roots of unity, and $e_N$ is the Weil pairing on $E[N]$ \footnote{Our definition of the Weil pairing is the reciprocal of Silverman's definition \cite[III.8]{silverman:AEC}. With our definition, we have $e_N(1/N,\tau/N)=e^{2\pi i/N}$ on the elliptic curve $\C/(\Z+\tau\Z)$ with $\mathrm{Im}(\tau)>0$.}.
\end{definition}
If $\phi : R \to R'$ is a ring morphism, we denote by $T_{R'} = (E_{R'},\omega_{R'},\beta_{R'})$ the base change of $T$ to $R'$ along $\phi$.

The isomorphism classes of test objects over $\C$ are in bijection with the set of lattices $L$ in $\C$ endowed with a symplectic basis of $\frac1N L/L$ \cite[2.4]{katz-interpolation}. Another example is given by the Tate curve $\mathrm{Tate}(q) = \mathbb{G}_m/q^\Z$ \cite[\S 8]{deligne:formulaire}. It is an elliptic curve over $\Z((q))$ endowed with the canonical differential $\omega_\mathrm{can}=dx/x$ and the level $N$ structure $\beta_\mathrm{can}(\zeta,n)= \zeta q^{n/N} \mod{q^\Z}$. The test object $(\mathrm{Tate}(q),\omega_\mathrm{can},\beta_\mathrm{can})$ is defined over $\Z((q^{1/N}))$.

\begin{definition} An \emph{algebraic modular form} of weight $k \in \Z$ and level $N$ over $R$ is the data, for each $R$-algebra $R'$, of a function
\begin{equation*}
F = F_{R'} : \{\textrm{isomorphism classes of test objects of level $N$ over $R'$}\} \to R'
\end{equation*}
satisfying the following properties:
\begin{enumerate}
\item \label{AlgMF cond 1} $F(E,\lambda^{-1} \omega,\beta) = \lambda^k F(E,\omega,\beta)$ for every $\lambda \in (R')^\times$;
\item \label{AlgMF cond 2} $F$ is compatible with base change: for every morphism of $R$-algebras $\psi : R' \to R''$ and for every test object $T$ of level $N$ over $R'$, we have $F_{R''}(T_{R''}) = \psi(F_{R'}(T))$.
\end{enumerate}
We denote by $M_k^{\mathrm{alg}}(\Gamma(N);R)$ the $R$-module of algebraic modular forms of weight $k$ and level $N$ over $R$.
\end{definition}
 Evaluating at the Tate curve provides an injective $R$-linear map
\begin{equation*}
M_k^{\mathrm{alg}}(\Gamma(N);R) \hookrightarrow \Z((q^{1/N})) \otimes_\Z R
\end{equation*}
called the $q$-expansion map. The \emph{$q$-expansion principle} states that if $R'$ is a subring of $R$, then an algebraic modular form $F \in M_k^\mathrm{alg}(\Gamma(N);R)$ belongs to $M_k^\mathrm{alg}(\Gamma(N);R')$ if and only if the $q$-expansion of $F$ has coefficients in $R'$.

Algebraic modular forms are related to classical modular forms as follows. To any algebraic modular form $F \in M_k^\mathrm{alg}(\Gamma(N);\C)$, we associate the function $F^\mathrm{an} : \h \to \C$ defined by
\begin{equation*}
F^\mathrm{an}(\tau) = F\Bigl(\frac{\C}{2\pi i\Z+2\pi i\tau\Z},dz,\beta_\tau \Bigr)
\end{equation*}
with $\beta_\tau(\zeta_N^m,n):=[2\pi i(m+n\tau)/N]$. 
\begin{proposition}
The map $F \mapsto F^\mathrm{an}$ induces an isomorphism between $M_k^\mathrm{alg}(\Gamma(N);\C)$ and the space $M_k^!(\Gamma(N))$ of weakly holomorphic modular forms on $\Gamma(N)$ (that is, holomorphic on $\h$ and meromorphic at the cusps). Moreover, the $q$-expansion of $F$ coincides with that of $F^\mathrm{an}$.
\end{proposition}

We now interpret the action of $\SL_2(\Z)$ on modular forms in algebraic terms. Let $F \in M_k^\mathrm{alg}(\Gamma(N);\C)$ with $f=F^\mathrm{an}$, and let $g = \sabcd abcd \in \SL_2(\Z)$. A simple computation shows that
\setcounter{equation}{2}
\begin{equation}\label{fg F}
(f |_k g)(\tau) = F\Bigl(\frac{\C}{2\pi i(\Z+\tau \Z)},dz,\beta'_\tau\Bigr)
\end{equation}
where the level $N$ structure $\beta'_\tau$ is given by
\begin{equation}\label{beta' beta}
\beta'_\tau(\zeta_N^m,n) = \beta_\tau(\zeta_N^{md+nb},mc+na).
\end{equation}
Let $\psi : (\Z/N\Z)^2 \to \mu_N(\C) \times \Z/N\Z$ be the isomorphism defined by $\psi(a,b)=(\zeta_N^b,a)$. Let us identify the level structure $\beta_\tau$ (resp. $\beta'_\tau$) with the map $\alpha_\tau = \beta_\tau \circ \psi$ (resp. $\alpha'_\tau = \beta'_\tau \circ \psi$). Then (\ref{beta' beta}) shows that
\begin{equation}
\alpha'_\tau(a,b)=\alpha_\tau((a,b)g).
\end{equation}
What we have here is the right action of $\SL_2(\Z)$ on the row space $(\Z/N\Z)^2$, which induces a left action on the set of level $N$ structures. As we will see, all this makes sense algebraically. For any $\Z[\zeta_N]$-algebra $R$, we denote by $\zeta_{N,R}$ the image of $\zeta_N=e^{2\pi i/N}$ under the structural morphism $\Z[\zeta_N] \to R$.

\begin{lemma}
If $R$ is a $\Z[\zeta_N,1/N]$-algebra, then there is an isomorphism of $R$-group schemes $(\Z/N\Z)_R \xrightarrow{\cong} (\mu_N)_R$ sending $1$ to $\zeta_{N,R}$.
\end{lemma}

\begin{proof}
Note that $(\mu_N)_R = \operatorname{Spec} R[t]/(t^N-1) = \operatorname{Spec} R[\Z/N\Z]$ and $(\Z/N\Z)_R = \operatorname{Spec} R^{\Z/N\Z}$. If $R=\C$, then $\C[\Z/N\Z] \cong \C^{\Z/N\Z}$ because all irreducible representations of $\Z/N\Z$ have dimension 1. This isomorphism $\mathcal{F}_\C$ is given by the Fourier transform, and both $\mathcal{F}_\C$ and $\mathcal{F}_\C^{-1}$ have coefficients in $\Z[\zeta_N,1/N]$ with respect to the natural bases. It follows that in general $R[\Z/N\Z] \cong R^{\Z/N\Z}$ and this isomorphism sends $[1]$ to $(\zeta_{N,R}^a)_{a \in \Z/N\Z}$.
\end{proof}

Let $R$ be a $\Z[\zeta_N,1/N]$-algebra. We have an isomorphism of $R$-group schemes
\begin{equation*}
\psi_R : (\Z/N\Z)^2_R \to (\mu_N)_R \times (\Z/N\Z)_R
\end{equation*}
given by $\psi_R(a,b)=(\zeta_{N,R}^b,a)$. The group $\SL_2(\Z)$ acts from the right on the row space $(\Z/N\Z)^2_R$ by $R$-automorphisms, and for $\alpha : (\Z/N\Z)^2_R \xrightarrow{\cong} E[N]$ we define
\begin{equation}\label{action g alpha}
(g \cdot \alpha)(a,b)=\alpha((a,b)g) \qquad ((a,b) \in (\Z/N\Z)^2).
\end{equation}
Using $\psi_R$, we transport this to a left action of $\SL_2(\Z)$ on the set of level $N$ structures of an elliptic curve over $R$. Given a test object $T=(E,\omega,\beta)$ over $R$, we define $g \cdot T := (E,\omega,g\cdot \beta)$. For any $F \in M_k^\mathrm{alg}(\Gamma(N);R)$, we define $F|g \in M_k^\mathrm{alg}(\Gamma(N);R)$ by the rule $(F|g)(T) = F(g \cdot T)$ for any test object $T$ over any $R$-algebra $R'$. The computation (\ref{fg F}) then shows that the right action of $\SL_2(\Z)$ on $M_k^{\mathrm{alg}}(\Gamma(N);\C)$ corresponds to the usual slash action on $M_k^!(\Gamma(N))$.

\begin{remark}\label{rem SL2 action}
The action of $\SL_2(\Z)$ on algebraic modular forms over $\Z[\zeta_N,1/N]$-algebras has the following consequence: if a classical modular form $f \in M_k(\Gamma(N))$ has Fourier coefficients in some subring $A$ of $\C$, then for any $g \in \SL_2(\Z)$, the Fourier expansion of $f|g$ lies in $\Z[[q^{1/N}]] \otimes A[\zeta_N,1/N]$.
\end{remark}

We now interpret the action of $\Aut(\C)$ in algebraic terms (see \cite[p. 88]{ohta}). Let $\sigma \in \Aut(\C)$. For any $\C$-algebra $R$, we define $R^\sigma := R \otimes_{\C,\sigma^{-1}} \C$, which means that $(ax) \otimes 1 = x \otimes \sigma^{-1}(a)$ for all $a \in \C$, $x \in R$. We endow $R^\sigma$ with the structure of a $\C$-algebra using the map $a \in \C \mapsto 1 \otimes a \in R^\sigma$. We denote by $\phi_\sigma : R \to R^\sigma$ the map defined by $\phi_\sigma(x)=x \otimes 1$. The map $\phi_\sigma$ is a ring isomorphism, but one should be careful that $\phi_\sigma$ is not a morphism of $\C$-algebras, as it is only $\sigma^{-1}$-linear. For any test object $T$ over $R$, we denote by $T^\sigma$ its base change to $R^\sigma$ using the ring morphism $\phi_\sigma$.

Let $F \in M_k^{\mathrm{alg}}(\Gamma(N);\C)$ be an algebraic modular form. For any $\C$-algebra $R$, we define
\begin{align*}
F^\sigma_R : \{\textrm{isomorphism classes of test objects of level $N$ over $R$}\} & \to R \\
T & \mapsto \phi_\sigma^{-1}\bigl(F_{R^\sigma}(T^\sigma)\bigr).
\end{align*}
One may check that the collection of functions $F^\sigma_R$ satisfies the conditions (\ref{AlgMF cond 1}) and (\ref{AlgMF cond 2}) above, hence defines an algebraic modular form $F^\sigma \in M_k^\mathrm{alg}(\Gamma(N);\C)$. Moreover, since the Tate curve is defined over $\Z((q))$, one may check that the map $F \mapsto F^\sigma$ corresponds to the usual action of $\Aut(\C)$ on the Fourier expansions of modular forms: for every $F \in M_k^\mathrm{alg}(\Gamma(N);\C)$ and every $\sigma \in \Aut(\C)$, we have $(F^\sigma)^\mathrm{an} = (F^\mathrm{an})^\sigma$.

We finally come to the proof of Theorem \ref{fg sigma}.

\begin{proof}
Let $f \in M_k(\Gamma(N))$ with corresponding algebraic modular form $F \in M_k^{\mathrm{alg}}(\Gamma(N);\C)$. Let $g \in \SL_2(\Z)$ and $\sigma \in \Aut(\C)$. We take as test object $T=(\mathrm{Tate}(q),\omega_\mathrm{can},\beta_{\mathrm{can}})$ over $R=\Z((q^{1/N})) \otimes \C$. Since a modular form is determined by its Fourier expansion, and unravelling the definitions of $F|g$ and $F^\sigma$, it suffices to check that the test objects $g \cdot T^\sigma$ and $(g_\lambda \cdot T)^\sigma$ over $R^\sigma$ are isomorphic. Since $\SL_2(\Z/N\Z)$ acts only on the level structures of the test objects, we have to show that
\begin{equation}\label{to show}
g \cdot \beta_\mathrm{can}^\sigma \cong (g_\lambda \cdot \beta_\mathrm{can})^\sigma.
\end{equation}
For any scheme $X$ over $R$, let $X^\sigma$ denote its base change to $R^\sigma$ along $\phi_\sigma$. Since $\phi_\sigma$ is a ring isomorphism, the canonical projection map $X^\sigma \to X$ is an isomorphism of schemes, and we also denote by $\phi_\sigma : X \to X^\sigma$ the inverse map.

Put $E=\mathrm{Tate}(q)$ and $\beta = \beta_\mathrm{can}$. Let $\alpha = \beta \circ \psi_R : (\Z/N\Z)^2_R \xrightarrow{\cong} E[N]$. By functoriality, the level structure $\beta^\sigma$ is given by the following commutative diagram
\begin{equation}\label{big diagram}
\begin{tikzcd}
(\Z/N\Z)^2_R \arrow[dotted]{d}[swap]{\gamma} \arrow{r}{\psi_R} \arrow[bend left=20]{rr}{\alpha} & (\mu_N)_R \times (\Z/N\Z)_R \arrow{r}{\beta} \arrow{d}{\phi_\sigma}[swap]{\cong} & E[N] \arrow{d}{\phi_\sigma}[swap]{\cong} \\
(\Z/N\Z)^2_{R^\sigma} \arrow{r}{\psi_{R^\sigma}} \arrow[bend right=20]{rr}{\alpha^\sigma} & (\mu_N)_{R^\sigma} \times (\Z/N\Z)_{R^\sigma} \arrow{r}{\beta^\sigma} & E^\sigma[N].
\end{tikzcd}
\end{equation}
Let us compute the dotted arrow $\gamma$. Since $\phi_\sigma$ is $\sigma^{-1}$-linear, we have $\phi_\sigma(\zeta_{N,R}) = \zeta_{N,R^\sigma}^{\lambda^{-1}}$. It follows that
\begin{equation}
\phi_\sigma(\psi_R(a,b)) = \phi_\sigma(\zeta_{N,R}^b,a) = (\zeta_{N,R^\sigma}^{\lambda^{-1}b},a) = \psi_{R^\sigma}(a,\lambda^{-1}b)
\end{equation}
so that $\gamma(a,b)=(a,\lambda^{-1}b)$. We may thus express $\alpha^\sigma$ in terms of $\alpha$ by
\begin{equation}\label{alpha sigma}
\alpha^\sigma(a,b) = \phi_\sigma \circ \alpha \circ \gamma^{-1}(a,b) = \phi_\sigma \circ \alpha(a,\lambda b) = \phi_\sigma \circ \alpha \biggl((a,b) \abcd 100\lambda \biggr).
\end{equation}
Let us make explicit both sides of (\ref{to show}). By (\ref{action g alpha}) and (\ref{alpha sigma}), the left hand side is given by
\begin{equation}
(g \cdot \alpha^\sigma)(a,b) = \alpha^\sigma((a,b) g) = \phi_\sigma \circ \alpha \biggl((a,b) g \abcd 100\lambda\biggr).
\end{equation}
Let us now turn to the right hand side of (\ref{to show}). By (\ref{action g alpha}), we have $(g_\lambda \cdot \alpha)(a,b) = \alpha((a,b)g_\lambda)$. Applying the commutative diagram (\ref{big diagram}) with $\alpha$ replaced by $g_\lambda \cdot \alpha$, we get
\begin{equation}
(g_\lambda \cdot \alpha)^\sigma(a,b) = \phi_\sigma \circ (g_\lambda \cdot \alpha) \biggl((a,b) \abcd 100\lambda \biggr) = \phi_\sigma \circ \alpha \biggl((a,b) \abcd 100\lambda g_\lambda \biggr).
\end{equation}
Finally, we note that $g \sabcd 100\lambda  = \sabcd 100\lambda g_\lambda$.

\end{proof}

\bibliographystyle{plain}
\bibliography{refs}

\vspace{\baselineskip}

\end{document}